\documentclass[12pt]{article}
\usepackage{a4wide}
\usepackage{amsmath,amssymb,amsthm}
\usepackage[T2A]{fontenc}
\usepackage{mathtools}
\usepackage{epsfig}
\usepackage{psfrag}

\newcommand{\X}[1][n]{X^{[#1]}}
\newcommand{\gr}{\Gamma}

\newcommand{\Pa}{\mathcal{P}}

\newcommand{\Qa}{\mathcal{Q}}
\newcommand{\St}{{\rm St}}
\newcommand{\RiSt}{{\rm RiSt}}
\newcommand{\Sym}{\mathop{\rm Sym}\nolimits}
\newcommand{\Aut}{\mathop{\rm Aut}\nolimits}

\author{Ievgen~V.~Bondarenko, Igor~O.~Samoilovych}
\title{\textbf{On finite generation of self-similar groups of finite type}}

\newtheorem{theorem}{Theorem}
\newtheorem{proposition}[theorem]{Proposition}
\newtheorem{corollary}[theorem]{Corollary}

\theoremstyle{definition}

\newtheorem{remark}{Remark}

\begin{document}
\maketitle

\begin{abstract}
A self-similar group of finite type is the profinite group of all automorphisms of a regular rooted
tree that locally around every vertex act as elements of a given finite group of allowed actions. We
provide criteria for determining when a self-similar group of finite type is finite, level-transitive,
or topologically finitely generated.  Using these criteria and GAP computations we show that for the
binary alphabet there is no infinite topologically finitely generated self-similar group given by
patterns of depth $3$, and there are $32$ such groups for depth~$4$.\\


\vspace{0.1cm} \noindent \textbf{Keywords}: self-similar group, finite generation,
branch group, profinite group

\end{abstract}

\section{Introduction}

There are two important classes of groups acting on regular rooted trees that have arisen as a generalization of the
Grigorchuk group: self-similar groups and branch groups. An automorphism group of a regular rooted tree is self-similar
if the restriction of the action of every its element onto every subtree can be given again by an element of the group.
There are many examples of self-similar groups with numerous extreme properties (like the Grigorchuk group) and this
class of groups is very promising for looking different counterexamples. At the same time, self-similar groups appear
naturally in many areas of mathematics and have strong connections with fractal geometry, dynamical systems, automata
theory (see \cite{self_sim_groups} and the references therein). Branch groups are automorphism groups of a tree whose
subgroup lattice is similar to the tree \cite{BGS:branch}. This class plays an important role in classification of
just-infinite groups \cite{G:just_inf}.

Self-similar groups of finite type have arisen as the closure of certain self-similar branch groups in
the topology of the tree. It was noticed in \cite[Section~7]{G:solved} that the closure of the
Grigorchuk group is a profinite self-similar group that can be described by a finite group of allowed
local actions on a finite tree (obtained from the binary tree by truncating at some depth).
R.I.\,Grigorchuk used this observation to define a self-similar group of finite type as the group of
all tree automorphisms that locally around every vertex act as elements of a given finite group (see
precise definition in the next section). The term ``group of finite type'' comes from the analogy with
shifts of finite type in symbolic dynamics \cite{symb_dyn} (note that a different term, namely finitely
constrained group, is used in \cite{S:Hausd,S:pattern}). Every self-similar group of finite type with
transitive action on levels of the tree is a profinite branch group by
\cite[Proposition~7.5]{G:solved}, and conversely, the closure (and profinite completion) of a
self-similar regular branch group with congruence subgroup property is a self-similar group of finite
type by \cite[Theorem~3]{S:Hausd}. The last observation was the main ingredient to compute the
Hausdorff dimension of such branch groups in \cite{S:Hausd}.

Although a self-similar group of finite type is easy to define by a finite group of patterns, it is not
clear what are the properties of the group. In particular, R.I.\,Grigorchuk asked in
\cite[Problem~7.3(i)]{G:solved} under what conditions a self-similar group of finite type is
topologically finitely generated. In this note we address this question and establish certain criterion
in Theorem~\ref{thm_fin_gen} as well as some necessary and sufficient conditions. We also answer such
basic questions like how to check whether a self-similar group of finite type is trivial, finite, or
acts transitively on levels of the tree.

The closure of the Grigorchuk group is a self-similar group of finite type defined by patters of depth $4$ over the
binary tree. The closure of groups defined in \cite{S:Hausd} give examples of infinite finitely generated self-similar
groups of finite type defined by patterns of depth $d$ for any $d\geq 4$. For depth $2$ and binary tree every
self-similar group of finite type is either finite or not finitely generated as shown in \cite{S:pattern}. The only
unknown case was for depth $3$. Using developed criteria and GAP computations we prove that there is no infinite finitely
generated self-similar group of finite type defined by patterns of depth $3$ over the binary tree. For depth $4$ there
are $32$ such groups (including the closures of the Grigorchuk group and the iterated monodromy group of $z^2+i$
\cite{GSS:z2+i}).

\section{Self-similar groups of finite type}

In this section we first recall all needed definitions and introduce self-similar groups of finite type
(see \cite{self_sim_groups,G:solved} for more information). After that we study conditions when a
self-similar group of finite type is trivial, finite, or level-transitive.

\vspace{0.2cm}\textbf{Tree $X^{*}$.} Let $X$ be a finite alphabet with at least two letters. Let $X^{*}$ be the free
monoid freely generated by $X$. The elements of $X^{*}$ are all finite words $x_1x_2\ldots x_n$ over $X$ (including the
empty word). We also use notation $X^{*}$ for the tree with the vertex set $X^{*}$ and edges $(v,vx)$ for all $v\in
X^{*}$ and $x\in X$. The set $X^n$ is the \textit{$n$-th level} of the tree $X^{*}$. The subtree of $X^{*}$ induced by
the set of vertices $\cup_{i=0}^n X^i$ is denoted by $\X[n]$.

Self-similar groups of finite type are defined as special subgroups of the group $\Aut X^{*}$ of all automorphisms of the
tree $X^{*}$. The group $\Aut X^{*}$ is profinite; it is the inverse limit of the sequence
\[
\ldots\rightarrow\Aut\X[3]\rightarrow\Aut\X[2]\rightarrow \Aut X,
\]
where the homomorphisms are given by restriction of the action.


\vspace{0.2cm}\textbf{Sections of automorphisms.} For every automorphism $g\in\Aut X^{*}$ and every word $v\in X^{*}$
define the \textit{section} $g_{(v)}\in \Aut X^{*}$ \textit{of $g$ at $v$} by the rule: $g_{(v)}(x)=y$ for $x,y\in X^{*}$
if and only if $g(vx)=g(v)y$. In other words, the section of $g$ at $v$ is the unique automorphism $g_{(v)}$ of $X^{*}$
such that $g(vx)=g(v)g_{(v)}(x)$ for all $x\in X^{*}$. Sections have the following properties:
\[
g_{(vu)}=g_{(v)(u)},\quad (g^{-1})_{(v)}=(g_{(g^{-1}(v))})^{-1},\quad (gh)_{(v)}=g_{(h(v))}h_{(v)}
\]
for all $v,u\in X^{*}$ (we are using left actions, i.e., $(gh)(v)=g(h(v))$).

A subgroup $G<\Aut X^{*}$ is called \textit{self-similar} if $g_{(v)}\in G$ for every $g\in G$ and $v\in X^{*}$.

The restriction of the action of an automorphism $g\in\Aut X^{*}$ to the subtree $\X[d]$ is denoted by
$g|_{\X[d]}\in\Aut\X[d]$. To every $g\in\Aut X^{*}$ there corresponds a collection $(g_{(v)}|_{\X[d]})_{v\in X^{*}}$ of
automorphisms from $\Aut\X[d]$ which completely describe the action of $g$ on the tree $X^{*}$.



\vspace{0.2cm}\textbf{Self-similar groups of finite type.} A subgroup $\Pa$ of $\Aut\X[d]$ will be
called a \textit{group of patterns of depth} $d$ (or a pattern group of depth $d$), $d\geq 1$. We say
that an automorphism $g\in\Aut X^{*}$ \textit{agrees with} $\Pa$ if every section $g_{(v)}$, $v\in
X^{*}$, acts on $\X[d]$ in the same way as some element in $\Pa$, i.e., $g_{(v)}|_{\X[d]}\in\Pa$ for
all $v\in X^{*}$. Since $\Pa$ is a group, the inverse $g^{-1}$ and all sections $g_{(v)}$ of such an
element $g$ agree with $\Pa$, the product of two elements that agree with $\Pa$ also agrees with $\Pa$.
We obtain the self-similar group $G_{\Pa}$ of all automorphisms $g\in\Aut X^{*}$ that agree with $\Pa$,
i.e., we define the group
\[
G_{\Pa}=\{ g\in\Aut X^{*} : g_{(v)}|_{\X[d]}\in\Pa \mbox{ for every } v\in X^{*}\},
\]
called the \textit{self-similar group of finite type given by the pattern group} $\Pa$. Note that Grigorchuk in
\cite{G:solved} introduced these groups using finite sets of forbidden patterns, while we are using ``allowed'' patterns.


Every group $G_{\Pa}$ is closed in the topology of $\Aut X^{*}$. Indeed, if for an element $g\in \Aut X^{*}$ the
restriction $g|_{\X[n]}$ belongs to $G_{\Pa}|_{\X[n]}$ for every $n\in\mathbb{N}$, then $g_{(v)}|_{\X[d]}\in \Pa$ for
every $v\in X^{*}$ and thus $g\in G_{\Pa}$. Hence $G_{\Pa}$ is a profinite group.

Let us consider a few simple examples. If $\Pa$ is trivial then $G_{\Pa}$ is trivial. If $\Pa=\Aut\X[d]$ then
$G_{\Pa}=\Aut X^{*}$ (for any $d\in\mathbb{N}$). For every subgroup $\Pa<\Sym(X)$ the infinitely iterated permutational
wreath product $\ldots\wr_X\Pa\wr_X\Pa$ is a self-similar group of finite type, where $\Pa$ is the corresponding group of
patterns of depth $1$, and every self-similar group of finite type given by patterns of depth $1$ is of this form.



\vspace{0.2cm}\textbf{Minimal pattern groups.} The same self-similar group of finite type may be given
by different groups of patterns of depth $d$ and we want to choose a unique pattern group in each
class. Let $G$ be a self-similar group of finite type given by a group of patterns of depth $d$ and
consider the pattern group $\Pa=G|_{\X[d]}$. Since the group $G$ is self-similar,
$g_{(v)}|_{\X[d]}\in\Pa$ for every $g\in G$ and $v\in X^{*}$, and thus $G<G_{\Pa}$. On the other hand,
it is clear from the definition that every pattern group of depth $d$ that produces $G$ contains $\Pa$
as a subgroup. Hence $G=G_{\Pa}$ and $\Pa$ is the smallest group of patterns of depth $d$ with this
property. A pattern group $\Pa$ of depth $d$ will be called \textit{minimal} if the equality
$G_{\Pa}=G_{\Qa}$ for $\Qa<\Aut\X[d]$ implies $\Pa<\Qa$. It follows from the above arguments that a
pattern group $\Pa$ of depth $d$ is minimal if and only if $\Pa=G_{\Pa}|_{\X[d]}$, in other words, if
every pattern in $\Pa$ is realized as a restriction of an element of $G_{\Pa}$. Every self-similar
group of finite type given by patterns of depth $d$ is represented by a unique minimal pattern group of
depth $d$.

\vspace{0.2cm}\textbf{Pattern graph.} Let $\Pa$ be a group of patterns of depth $d$. In order to
minimize $\Pa$ one can use a directed labeled graph $\Gamma_{\Pa}$ which we call the \textit{pattern
graph} associated to $\Pa$. The vertices of $\Gamma_{\Pa}$ are the elements of $\Pa$ and for
$a,b\in\Pa$ and $x\in X$ we put a labeled arrow $a\stackrel{x}{\rightarrow}b$ whenever
$a_{(x)}|_{\X[d-1]}=b|_{\X[d-1]}$. Informally, the arrow $a\stackrel{x}{\rightarrow}b$ shows that we
can use the pattern $b$ to extend the action of $a$ on the subtree $x\X[d]$ (see
Fig.~\ref{fig_PatternGraph}). If a vertex $a\in\Pa$ does not have an outgoing edge labeled by $x$ for
some letter $x\in X$, then the action of $a$ cannot be extended to the next level using patterns from
$\Pa$; in other words $a$ is not a restriction of an element of $G_{\Pa}$. Now it is clear how to
minimize $\Pa$: we just remove every vertex that does not have an outgoing edge labeled by $x$ for some
$x\in X$ and repeat this reduction as long as possible. The remaining patterns will form a minimal
pattern group for $G_{\Pa}$. In particular, $\Pa$ is minimal if and only if every vertex of $\gr_{\Pa}$
has an outgoing edge labeled by $x$ for every $x\in X$.

\begin{figure}[ph]
\begin{center}
\psfrag{x}{$x$} \psfrag{b}{$b$} \psfrag{a}{$a$} \psfrag{GP}{in $\Gamma_{\Pa}$}
\epsfig{file=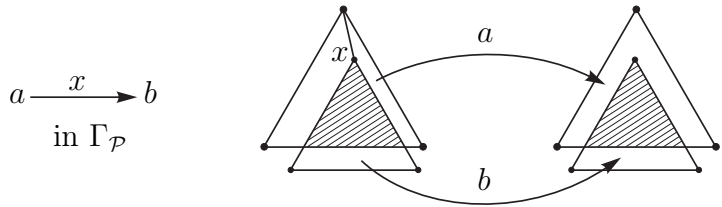,width=270pt}\caption{Coordination between patterns.}
\label{fig_PatternGraph}
\end{center}
\end{figure}

The graph $\Gamma_{\Pa}$ can be used to represent elements of the group $G_{\Pa}$ by graph homomorphisms as follows. Let
us take the tree $X^{*}$ and add direction and label to every edge by $v\stackrel{x}{\rightarrow}vx$ for every $v\in
X^{*}$ and $x\in X$. Then every element $g\in G_{\Pa}$ defines a homomorphism $\phi:X^{*}\rightarrow \Gamma_{\Pa}$ of
labeled directed graphs by the rule $\phi(v)=g_{(v)}|_{\X[d]}$. Indeed, for every arrow $v\stackrel{x}{\rightarrow}vx$ in
the tree $X^{*}$ the elements $g_{(v)(x)}$ and $g_{(vx)}$ are the same and we have the arrow
$\phi(v)\stackrel{x}{\rightarrow}\phi(vx)$ in the graph $\gr_{\Pa}$. And vise versa, every homomorphism
$\phi:X^{*}\rightarrow \Gamma_{\Pa}$ defines an element $g\in G_{\Pa}$ by its restrictions $g_{(v)}|_{\X[d]}=\phi(v)$.
This description is an analog of a standard statement in symbolic dynamics that every shift of finite type is sofic (see
\cite[Theorem~3.1.5]{symb_dyn}), and pattern graphs play a role of recognition graphs. One can use this observation to
introduce the notion of a self-similar group of sofic type which we will discuss elsewhere.

\vspace{0.2cm}\textbf{Branching properties.} Let us explain the connection between self-similar groups of finite type and
branch groups mentioned in Introduction.

Let $G$ be a subgroup of $\Aut X^{*}$. The \textit{vertex stabilizer} $\St_G(v)$ of a vertex $v\in X^{*}$ is the subgroup
of all $g\in G$ such that $g(v)=v$. The \textit{$n$-th level stabilizer} $\St_G(n)$ is the subgroup of all $g\in G$ such
that $g(v)=v$ for every $v\in X^n$. Notice that $\St_G(v)$ and $\St_G(n)$ have finite index in $G$. The \textit{rigid
vertex stabilizer} $\RiSt_G(v)$ of a vertex $v\in X^{*}$ is the subgroup of all $g\in G$ such that $g(u)=u$ for every
vertex $u\in X^{*}\setminus vX^{*}$. The set of all sections $g_{(v)}$ for $g\in\RiSt_G(v)$ forms a group which we call
the \textit{section group} of $\RiSt_G(v)$ at the vertex $v$. The group $G$ is called \textit{level-transitive} if it
acts transitively on all levels $X^n$ of the tree. The group $G$ is called \textit{regular branch} branching over its
subgroup $K$ if $G$ is level-transitive, $K$ is a normal subgroup of finite index, and the group of all automorphism
$g\in\St_{\Aut X^{*}}(1)$ such that the tuple $(g_{(x)})_{x\in X}$ belongs to $\prod_X K$ is a subgroup of finite index
in $K$. Note that the last condition implies that the section group of $\RiSt_K(v)$ at $v$ contains $K$ for every vertex
$v\in X^{*}$.

Every level-transitive self-similar group $G_{\Pa}$ of finite type given by patterns of depth $d$ is
regular branch over its level stabilizer $\St_{G_{\Pa}}(d-1)$ (see \cite[Proposition~7.15]{G:solved}).
Indeed, notice that for every element $h\in\St_{G_\Pa}(d-1)$ and any vertex $v\in X^{*}$ the unique
automorphism $g\in\RiSt_{\Aut X^{*}}(v)$ such that $g_{(v)}=h$ agrees with the pattern group $\Pa$ and
hence belongs to $G_{\Pa}$. It follows that $\St_{G_{\Pa}}(n)$ for $n\geq d$ decomposes into the direct
product
\[
\St_{G_{\Pa}}(n)\cong \St_{G_{\Pa}}(d-1)\times \ldots \times \St_{G_{\Pa}}(d-1)
\]
of $|X|^{n-d+1}$ copies of $\St_{G_{\Pa}}(d-1)$, where each factor acts on the corresponding subtree $vX^{*}$ for $v\in
X^{n-d+1}$. The last condition in the definition of a regular branch group follows. Conversely, if $G$ is a self-similar
regular branch group branching over its level stabilizer $\St_G(d-1)$ then the closure of $G$ in $\Aut X^{*}$ is a
self-similar group of finite type given by patterns of depth $d$ (see \cite[Theorem~3]{S:Hausd}).

\vspace{0.2cm}\textbf{Triviality, finiteness, and level-transitivity of $G_{\Pa}$.} Given a pattern
group $\Pa$ we want to understand whether the group $G_{\Pa}$ is trivial, finite, or acts transitively
on the levels of the tree $X^{*}$. The answer to the question about triviality of $G_{\Pa}$ directly
follows from the definition of a minimal pattern group. Namely, the group $G_{\Pa}$ is trivial if and
only if minimizing $\Pa$ we obtain the trivial group.

The finiteness of $G_{\Pa}$ can be effectively checked using the next statement.

\begin{proposition}\label{prop_finite}
Let $\Pa$ be a minimal pattern group of depth $d$. The group $G_{\Pa}$ is finite if and only if the stabilizer
$\St_{\Pa}(d-1)$ is trivial, and in this case $G_{\Pa}$ is isomorphic to $\Pa$.
\end{proposition}
\begin{proof}
Let $\gr_{\Pa}$ be the pattern graph of $\Pa$ and put $m=|\St_{\Pa}(d-1)|$. Notice that $(bc)|_{\X[d-1]}=b|_{\X[d-1]}$
for every $b\in\Pa$ and $c\in\St_{\Pa}(d-1)$. Hence if $a\stackrel{x}{\rightarrow}b$ is an arrow in $\gr_{\Pa}$ then
$a\stackrel{x}{\rightarrow}bc$ is also an arrow in $\gr_{\Pa}$ for every $c\in\St_{\Pa}(d-1)$, and every outgoing arrow
at $a$ with label $x$ is of this form. Therefore, since $\Pa$ is minimal, every vertex of $\gr_{\Pa}$ has precisely $m$
outgoing edges labeled by $x$ for every $x\in X$. It follows that for every $a\in\Pa$ there are precisely $m^{|X|}$
elements $g\in G_{\Pa}|_{\X[d+1]}$ such that $g|_{\X[d]}=a$; in other words, every pattern in $\Pa$ has $m^{|X|}$
extensions to the next level. Then for each level $n>d$ and for every $f\in G_{\Pa}|_{\X[n]}$ there are precisely
$m^{|X|^{n-d+1}}$ elements $g\in G_{\Pa}|_{\X[n+1]}$ such that $g|_{\X[n]}=f$. Now we can compute the total number of
elements in the restriction $G_{\Pa}|_{\X[n]}$:
\[
|G_{\Pa}|_{\X[n]}|=|\Pa|\cdot m^{|X|+|X|^2+\ldots+|X|^{n-d}}, \mbox{ for } n>d.
\]
Therefore the group $G_{\Pa}$ is finite if and only if $m=1$, i.e., when the group $\St_{\Pa}(d-1)$ is trivial. In this
case, $|G_{\Pa}|=|\Pa|$ and the restriction $g\mapsto g|_{\X[d]}$ is an isomorphism between $G_{\Pa}$ and $\Pa$.
\end{proof}

It follows from the proof that we can also use the pattern graph $\Gamma_{\Pa}$ to check the finiteness of $G_{\Pa}$. If
$\Pa$ is minimal, then the group $G_{\Pa}$ is finite if and only if some (equivalently, every) vertex of $\Gamma_{\Pa}$
has only one outgoing edge labeled by $x$ for each $x\in X$.

Let us treat transitivity on levels. We will use the standard observation that a subgroup $G<\Aut X^{*}$ acts
transitively on $X^{n+1}$ if and only if it acts transitively on $X^{n}$ and the stabilizer $\St_G(v)$ of some (every)
vertex $v\in X^n$ acts transitively on~$vX$.

Let $\Pa$ be a minimal pattern group of depth $d$ and consider the self-similar group of finite type $G_{\Pa}$. We fix a
letter $x\in X$ and use notation $x^n$ for the word $x\ldots x$ ($n$ times).  Let $\Pa_n$ be the group of all elements
$a\in\Pa$ for which there exists $g\in\St_{G_{\Pa}}(x^n)$ such that $g_{(x^n)}|_{\X[d]}=a$. Then $\St_{G_{\Pa}}(x^n)$ is
transitive on $x^nX$ if and only if $\Pa_n$ is transitive on $X$. It follows that $G_{\Pa}$ is level-transitive if and
only if each group $\Pa_n$ for $n\geq 0$ is transitive on $X$. Notice that the groups $\Pa_n$ can be computed recursively
by the rule: $\Pa_0=\Pa$ and
\[
\Pa_{n+1}=\left\{a\in\Pa_n : \mbox{ there exists } b\in\St_{\Pa_n}(x) \mbox{ such that }
b_{(x)}|_{\X[d-1]}=a|_{\X[d-1]}\right\}.
\]
We obtain a decreasing sequence $\Pa>\Pa_1>\ldots$ of finite groups which should stabilize on some
subgroup $\Qa<\Pa$, $\Qa=\cap_{n\geq 0} \Pa_n$. Moreover, if we take the smallest $n$ such that
$\Pa_n=\Pa_{n+1}$ then $\Pa_n=\Pa_{n+k}$ for every $k\in\mathbb{N}$, and thus $\Qa=\Pa_n$. Hence the
group $\Qa$ can be algorithmically computed. We have proved the following effective criterium.

\begin{proposition}\label{prop_transitive}
The group $G_{\Pa}$ is level-transitive if and only if the group $\Qa$ is transitive on $X$.
\end{proposition}

\section{Finite generation of groups $G_{\Pa}$}

In this section we study when the group $G_{\Pa}$ is topologically finitely generated. Further we omit the word
``topologically''.

\begin{theorem}\label{thm_fin_gen}
Let $G$ be a level-transitive self-similar group of finite type given by patterns of depth $d$. The
group $G$ is finitely generated if and only if there exists $n\geq d$ such that the commutator of
$\St_{G}(d-1)|_{\X[n]}$ contains $\St_{G}(n-1)|_{\X[n]}$.
\end{theorem}
\begin{proof}
Let $G=G_\Pa$ for a minimal pattern group $\Pa$ of depth $d$.

First we prove the necessity. The proof will not use transitivity on levels. Assume that the commutator
of $\St_{G}(d-1)|_{\X[n]}$ does not contain $\St_{G}(n-1)|_{\X[n]}$ for every $n\geq d$. Let us prove
that $\St_{G}(d-1)$ and thus $G$ are not finitely generated. In the proof we will use notations
$S=\St_{G}(d-1)$ and $S_n=S|_{\X[n]}$. For each $m\geq d$ consider the homomorphism
\[
\varphi:S\rightarrow \prod_{n=d}^m S_n/[S_n,S_n], \ \varphi(g)=(g|_{\X[n]} [S_n,S_n])_{n=d}^m.
\]
Recall that the stabilizer $\St_{G}(n-1)|_{\X[n]}$ decomposes into the direct product
$\St_{\Pa}(d-1)\times\ldots\times\St_{\Pa}(d-1)$ of $|X|^{n-d}$ copies of $\St_{\Pa}(d-1)$. By our
assumption there exists an element $g_n=(1,\ldots,a_n,\ldots,1)$, $a_n\in\St_{\Pa}(d-1)$, of this
product that does not belong to the commutator $[S_n,S_n]$. Let $A_n$ be the group generated by the
image of $g_n$ in the quotient of $\St_{G}(n-1)|_{\X[n]}$ by $[S_n,S_n]$. The group $A_n$ is a
nontrivial subgroup of the finite abelian group $S_n/[S_n,S_n]$. Hence $A_n$ is also a quotient of
$S_n/[S_n,S_n]$. Composing with $\varphi$ we obtain a homomorphism from $S$ to $\prod_{n=d}^m A_n$.
Moreover, for $i<n$ the $i$-th component of the image of $g_n$ in this direct product is trivial. It
follows that $\prod_{n=d}^m A_n$ is a homomorphic image of $S$. Since $|A_n|\leq |\Pa|$ for all $n$,
the number of generators of $\prod_{n=d}^m A_n$ goes to infinity as $m$ goes to infinity. Hence $S$ is
not finitely generated.

Let us prove the converse. Fix $k\geq d$ such that the commutator of $\St_{G}(d-1)|_{\X[k]}$ contains
$\St_{G}(k-1)|_{\X[k]}$. We construct a finitely generated dense subgroup of $G$ using the techniques
from branch groups (see \cite{BGS:branch,B:fgwreath}). Let $f_1,\ldots,f_l$ and $h_1,\ldots,h_m$ be the
elements of $G$ such that
\[
\langle f_1,\ldots,f_l\rangle |_{\X[1+d+k]}=G|_{\X[1+d+k]}\ \mbox{ and } \ \langle
h_1,\ldots,h_m\rangle|_{\X[k]}=\St_G(d-1)|_{\X[k]}.
\]
The group $\St_{\Pa}(d-1)$ is nontrivial by Proposition~\ref{prop_finite}, and we can find $v\in X^d$ and
$a\in\St_G(d-1)$ such that $a(v)\neq v$ (the element $a$ will be used to shift the section of certain automorphisms at
the vertex $v$). Fix two letter $x,y\in X$, $x\neq y$. Define the automorphisms $g_1,\ldots,g_m$ recursively by their
sections:
\[
{g_i}_{(yv)}=h_i\ \mbox{ and }\ {g_i}_{(x)}=g_i, \ i=1,\ldots,m,
\]
and the other sections are trivial (see Fig.~\ref{fig_Automorphisms}). Notice that $g_1,\ldots,g_m$ belong to $G$.

\begin{figure}
\psfrag{gi}{$g_i$} \psfrag{hi}{$h_i$} \psfrag{x}{$x$} \psfrag{y}{$y$} \psfrag{yv}{$yv$} \psfrag{xy}{$xy$}
\psfrag{x2}{$x^2$} \psfrag{x2y}{$x^2y$} \psfrag{x2yv}{$x^2yv$} \psfrag{xyv}{$xyv$} \psfrag{g}{$g$} \psfrag{a}{$a$}
\centerline{\epsfig{file=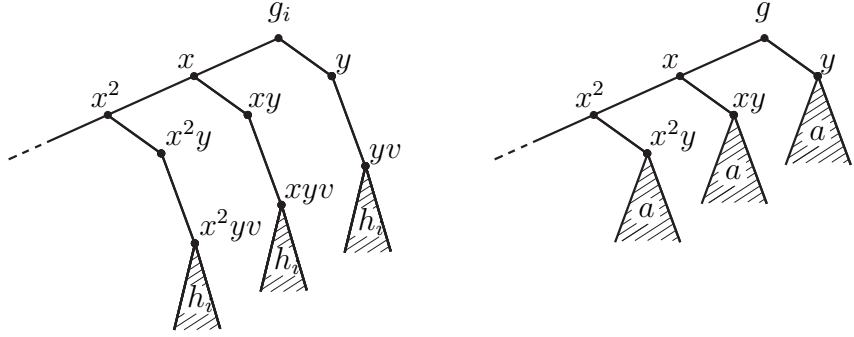,width=320pt}} \vspace*{8pt} \caption{The construction of generator $g_i$ and
conjugator $g$.}\label{fig_Automorphisms}
\end{figure}

Consider the group $H=\langle f_1,\ldots,f_l, h_1,\ldots,h_m, g_1,\ldots,g_m\rangle$ and let us show that $H$ is dense in
$G$. We need to prove that $H|_{\X[n]}=G|_{\X[n]}$ for all $n\in\mathbb{N}$. The statement holds for $n\leq 1+d+k$ by
construction. By induction on $n$ assume that we have proved it for all levels $\leq n+d+k$. There exists an element
$g\in G$ such that
\[
g_{(x^iy)}=a \mbox{ for } i=0,\ldots,n-1
\]
and the other sections are trivial (see Fig.~\ref{fig_Automorphisms}). By inductive hypothesis there exists $h\in H$ such
that $h|_{\X[n+d+k]}=g|_{\X[n+d+k]}$. Then the commutator $[h^{-1}g_ih,g_j]$ acts trivially on the vertices in
$\X[n+d+k+1]\setminus x^nyv\X[k]$ and at the vertex $x^nyv$ has section
\[
[h^{-1}g_ih,g_j]_{(x^nyv)}=[h^{-1}_{(x^nyv)}h_ih_{(x^nyv)},h_j].
\]
Conjugating by generators $g_1,\ldots,g_m$ we obtain that the section group of
$\RiSt_H(x^nyv)|_{\X[n+d+k+1]}$ at $x^nyv$ contains the commutator of $\St_G(d-1)|_{\X[k]}$ and hence
$\St_G(k-1)|_{\X[k]}$. Since the action is transitive this holds for every vertex of the level
$X^{n+d+1}$. Hence $\St_H(n+d+k)|_{\X[n+d+k+1]}=\St_G(n+d+k)|_{\X[n+d+k+1]}$ and the statement follows.

\end{proof}

\begin{remark}
An automorphism of the tree $X^{*}$ is called \textit{finite-state} if it has finitely many sections
(the term comes from automata theory); a subgroup is finite-state if it consists of finite-state
automorphisms. We can always choose elements $f_1,\ldots,f_l$ and $h_1,\ldots,h_m$ so that they are
finite-state. Then the elements $g_1,\ldots,g_m$ and the group $H$ constructed in the proof will be
also finite-state. Adding sections of elements we obtain a finitely generated finite-state self-similar
dense subgroup in $G$.
\end{remark}

\begin{remark}
The condition of level-transitivity cannot be dropped in Theorem~\ref{thm_fin_gen}. For example, consider the alternating
group $A_5$ with the natural action on $\{1,2,3,4,5\}$, extend the action to the alphabet $X=\{0,1,2,3,4,5\}$ by putting
$\pi(0)=0$ for every $\pi\in A_5$, and consider the infinitely iterated permutational wreath product $G_{A_5}=\ldots
\wr_X A_5\wr_X A_5$. The group $A_5$ is perfect, i.e., $[A_5,A_5]=A_5$, hence the condition in Theorem~\ref{thm_fin_gen}
holds for $n=d=1$. However the group $G$ is not finitely generated, because the map $g\mapsto
(g_{(0^n)}|_{X})_{n\in\mathbb{N}}$ is a surjective homomorphism from $G$ to the product $\prod_{\mathbb{N}} A_5$ which is
not finitely generated.
\end{remark}

\begin{remark}
It is not difficult to see that for a group $G_{\Pa}$ given by a transitive pattern group $\Pa$ of depth $1$ the
condition in the theorem holds for some $n$ if and only if the group $\Pa$ is perfect. Hence Theorem~\ref{thm_fin_gen}
generalizes Corollary~3.6 in \cite{B:fgwreath} about finite generation of iterated permutational wreath products
$\ldots\wr_X\Pa\wr_X\Pa$.
\end{remark}

\begin{proposition}\label{prop_not_fin_gen}
Let $G$ be a self-similar group of finite type given by patterns of depth $d$. If there exists $n\geq
d$ such that the commutator of $G|_{\X[n]}$ does not contain $\St_{G}(n-1)|_{\X[n]}$ then the group $G$
is not finitely generated.
\end{proposition}
\begin{proof}
The proof uses the same arguments as in the first part of the proof above. Fix $n\geq d$ such that the
commutator of $G_n:=G|_{\X[n]}$ does not contain $\St_{G}(n-1)|_{\X[n]}$. For every $k\in\mathbb{N}$
consider the map
\[
\varphi_k:G\rightarrow G_n/[G_n,G_n], \quad \varphi_k(g)=\prod_{v\in X^k}g_{(v)}|_{\X[n]} [G_n,G_n].
\]
Since $G_n/[G_n,G_n]$ is abelian every map $\varphi_k$ is a homomorphism. Now for every $m\in\mathbb{N}$ consider the
homomorphism $\varphi:G\rightarrow \prod_{k=1}^m G_n/[G_n,G_n]$, $\varphi(g)=(\varphi_k(g))_{k=1}^m$. For every $k$ and
every pattern $a\in\St_{G}(n-1)|_{\X[n]}$ there exists $g$ in the rigid stabilizer $\RiSt_{G}(v)$ of a vertex $v\in X^k$
such that $g_{(v)}|_{\X[n]}=a$, and thus $\varphi_k(g)=a$ and $\varphi_i(g)=e$ for $i<k$. Since
$\St_{G}(n-1)|_{\X[n]}/[G_n,G_n]$ is a homomorphic image of $G_n$, it follows that the abelian group $\prod_{k=1}^m
\St_{G}(n-1)|_{\X[n]}/[G_n,G_n]$ is a homomorphic image of $G$ for every $m$. Hence $G$ is not finitely generated.
\end{proof}

The next statement generalizes Proposition~2 in \cite{S:pattern}.

\begin{corollary}\label{cor_abelian}
Let $\Pa$ be an abelian pattern group. The group $G_{\Pa}$ is finitely generated if and only if it is finite.
\end{corollary}
\begin{proof}
The statement follows from Proposition~\ref{prop_finite} and Proposition~\ref{prop_not_fin_gen} with $n=d$.
\end{proof}

\begin{corollary}
Take a cyclic subgroup $C<\Sym(X)$ and consider the group $C\wr_X C$ as a natural subgroup of $\Aut\X[2]\cong
\Sym(X)\wr_X\Sym(X)$. Then for any nilpotent pattern group $\Pa<C\wr_X C$ the group $G_{\Pa}$ is finitely generated if
and only if it is finite.
\end{corollary}
\begin{proof}
Since $\Pa/\St_{\Pa}(1)$ is cyclic, the commutator $[\Pa,\Pa]$ is a subgroup of $\St_{\Pa}(1)$. If it
is a proper subgroup then the group $G_{\Pa}$ is not finitely generated by
Proposition~\ref{prop_not_fin_gen}. Suppose $[\Pa,\Pa]=\St_{\Pa}(1)$. For any $a,b\in\Pa$ there exists
$k\in\mathbb{N}$ such that $a^kb$ or $b^ka$ belongs to $\St_{\Pa}(1)$. Using the equality
$[a,b]=[a,a^kb]=[b^ka,b]$ we obtain that $[\Pa,\Pa]=[\Pa,\St_{\Pa}(1)]$. Since $\Pa$ is nilpotent, the
last equality implies that $[\Pa,\Pa]=\St_{\Pa}(1)=\{1\}$ and hence the group $G_{\Pa}$ is finite by
Proposition~\ref{prop_finite}.
\end{proof}

\section{A few classification results}\label{sec_Examples}

In this section we classify self-similar groups of finite type for the binary alphabet $X=\{0,1\}$ and
depth $\leq 4$. All computations were made in GAP. Our strategy for classifying self-similar groups of
finite type of a given depth $d$ is the following. First we find all subgroups in $\Aut\X[d]$, then
minimize all subgroups and obtain the number of all minimal pattern groups, which is equal to the
number of self-similar groups of finite type of a given depth as subgroups in $\Aut X^{*}$. Further we
distinguish all finite groups using Proposition~\ref{prop_finite}. Then we apply
Proposition~\ref{prop_not_fin_gen} for small values of $n$ to distinguish groups that are not finitely
generated. An infinite self-similar group over the binary alphabet is level-transitive (see
\cite[Lemma~3]{classif}), hence the rest of the groups are level-transitive and we can apply
Theorem~\ref{thm_fin_gen}. In this way it was possible to obtain the following results.

\textbf{Depth $d=2$.} This case was treated in \cite{S:pattern}. There are ten subgroups in
$\Aut\X[2]$, six minimal pattern subgroups, and hence six self-similar groups of finite type. Among
them there are three finite groups, namely the trivial group and two groups isomorphic to $C_2$, and
the other three groups are not finitely generated (Proposition~\ref{prop_not_fin_gen} works with
$n=2$).

\textbf{Depth $d=3$.} There are $576$ subgroups in $\Aut\X[3]$, $60$ minimal pattern subgroups, and
hence $60$ self-similar groups of finite type. Among them there are $23$ finite groups, namely the
trivial group, two groups isomorphic to $C_2$, four groups isomorphic to $C_2\times C_2$, $16$ groups
isomorphic to the dihedral group $D_8$. The other $37$ groups are not finitely generated ($27$ groups
satisfy Proposition~\ref{prop_not_fin_gen} with $n=3$ and $10$ groups with $n=4$).

\begin{corollary}
A self-similar group of finite type given by patterns of depth $d\leq 3$ over the binary alphabet is
either finite or not finitely generated.
\end{corollary}

\textbf{Depth $d=4$.} There are $4544$ self-similar groups of finite type. Among them there are $1535$
finite groups, namely the trivial group, two groups isomorphic to $C_2$, four groups isomorphic to
$C_2\times C_2$, $16$ groups isomorphic to $D_8$, eight groups isomorphic to $C_2\times C_2\times C_2$,
$96$ groups isomorphic to $C_2\times D_8$, $128$ groups isomorphic to $(C_2\times C_2\times C_2\times
C_2) \rtimes C_2$, $256$ groups isomorphic to $(((C_4\times C_2) \rtimes C_2) \rtimes C_2) \rtimes
C_2$, and $1024$ groups isomorphic to $\Aut\X[3]\cong C_2\wr_X C_2\wr_X C_2$. Among the rest of the
groups there are $2977$ not finitely generated ($1235$ groups satisfy
Proposition~\ref{prop_not_fin_gen} with $n=4$, $778$ groups with $n=5$, $508$ groups with $n=6$, $200$
groups with $n=7$, and $256$ groups with $n=8$) and $32$ finitely generated groups that satisfy
Theorem~\ref{thm_fin_gen} with $n=6$. The pattern groups of these $32$ self-similar groups of finite
type all have order $4096$, their restriction on $\X[3]$ is equal to $\Aut\X[3]$, and among them there
are $20$ pairwise non-isomorphic groups. These pattern groups can be described as follows. Let us
consider the group $\Aut\X[4]$ as a natural subgroup of the symmetric group~$\Sym(16)$ on the set
$\{1,2,\ldots,16\}\leftrightarrow X^4$ and fix the permutations:
\[\begin{array}{l}
a_1=(1,9)(2,10)(3,11)(4,12)(5,13)(6,14)(7,15)(8,16)\\
a_2=(1,10,2,9)(3,11)(4,12)(5,14,6,13)(7,15)(8,16)\\
a_3=(1,10)(2,9)(3,11)(4,12)(5,13)(6,14)(7,15)(8,16)\\
a_4=(1,9,2,10)(3,11)(4,12)(5,14,6,13)(7,15)(8,16)
\end{array}\]
\[\begin{array}{lll}
b_1=(1,5)(2,6)(3,7)(4,8)(9,10) \quad & c_1=(1,3)(2,4) & c_3=(1,3)(2,4)(5,6)\\
b_2=(1,6)(2,5)(3,7)(4,8)(9,10) \quad & c_2=(1,4,2,3) & c_4=(1,4,2,3)(5,6)
\end{array}
\]
Then the $32$ pattern groups mentioned above is the family of groups $\Pa_{ijk}=\langle
a_i,b_j,c_k\rangle$. In this family: the self-similar group of finite type $G_{\Pa_{123}}$ is the
closure of the Grigorchuk group and $G_{\Pa_{111}}$ is the closure of the iterated monodromy group of
$z^2+i$ \cite{GSS:z2+i}.

\end{document}